\documentclass[runningheads,a4paper]{llncs}
\usepackage[T1]{fontenc}
\usepackage[latin1]{inputenc}
\usepackage{amsmath, amssymb}
\usepackage[textwidth=3.3cm]{todonotes}
\usepackage{float}
\usepackage{pgf, tikz}
\usetikzlibrary{arrows, patterns, snakes,calc}
\usepackage{graphicx}
\usepackage{url}
\usepackage{enumerate}
\usepackage{ifthen}

\newcommand{\pal}{\ensuremath\text{Pal}}

\newcommand{\T}{\ensuremath{\mathcal T}}
\renewcommand{\P}{\ensuremath{\mathcal P}}
\DeclareMathOperator{\C}{\mathcal{C}}

\DeclareMathOperator{\R}{\mathcal{R}}
\newcommand{\lang}{\ensuremath\mathcal{L}}
\newcommand{\bigo}{\ensuremath\mathcal{O}}

\newcommand{\m}[1]{\ensuremath\widetilde{#1}}
\newcommand{\tailleCercle}{4pt}
\newcommand{\tailleNoeud}{6pt}
\newcommand{\tikzScale}{0.6}

\newtheorem{thm}{Theorem}
\newtheorem{prop}[thm]{Proposition}

\newtheorem{lem}[thm]{Lemma}

\renewenvironment{proof}{\paragraph{Proof.}}{\hfill$\square$\\}

\DeclareMathOperator{\Pal}{Pal}

\definecolor{lgray}{gray}{0.7}

\title{Palindromic complexity of trees}
\author{Sre\v{c}ko Brlek$^1$
\and Nadia Lafrenière$^1$
\and Xavier Provençal$^2$
}

\urldef{\mailsa}\path|brlek.srecko@uqam.ca|
\urldef{\mailsb}\path|lafreniere.nadia.2@courrier.uqam.ca|
\urldef{\mailsc}\path|xavier.provencal@univ-savoie.fr|

\institute{${}^1$Université du Québec à Montréal, Montréal, Québec, Canada\\
$^2$Université de Savoie, Chambéry, France\\
\mailsa \\
\mailsb \\
\mailsc
}

\authorrunning{Brlek, Lafrenière, Provençal}

\newcommand{\keywords}[1]{\par\addvspace\baselineskip
\noindent\keywordname\enspace\ignorespaces#1}

\begin{document}
\maketitle

\sloppy

\begin{abstract}
 We consider finite trees with edges labeled by letters on a finite alphabet $\varSigma$. Each  pair  of nodes defines a unique labeled path whose trace is a word of the free monoid $\varSigma^*$. The set of all such words defines the language of the tree. In this paper, we investigate the palindromic complexity of trees and provide hints for an upper bound on the number of distinct palindromes in the language of a tree.
\end{abstract}
\section{Introduction}
The palindromic language of a word has been extensively investigated recently, see for instance \cite{abcd} and  more recently \cite{PrBrRe,BrRe}.  In particular, Droubay, Justin and
Pirillo \cite{DrJuPi} established the following property:

\begin{thm}[\rm Proposition 2 \cite{DrJuPi}]\label{thm DrJuPi}
  A word $w$ contains at most $|w|+1$ distinct palindromes.
\end{thm}

Several families of words have been studied for their total palindromic
complexity, among which periodic words \cite{bhnr}, fixed points of morphism
\cite{hks} and Sturmian words \cite{DrJuPi}. \\

Considering words as geometrical objects, we can extend some definitions. For example, the notion of palindrome appears in the study of multidimensional geometric structures, thus introducing a new characterization. Some known classes of words are often redefined as digital planes \cite{BV,LR}, and the adjacency graph of structures obtained by symmetries appeared more recently \cite{DV12}. In the latter article, authors show that the obtained graph is a tree and its palindromes have been described by Domenjoud, Proven\c{c}al and Vuillon \cite{DPV}.
The trees studied by Domenjoud and Vuillon \cite{DV12} are obtained by iterated
palindromic closure, just as Sturmian \cite{deLu} and episturmian
\cite{DrJuPi,GJ} words. It has also been shown \cite{DPV} that the total number
of distinct nonempty palindromes in these trees is equal to the number of
edges in the trees. This property highlights the fact that these trees form a
multidimensional generalization of Sturmian words. 

A finite word is identified with a tree made of only one branch. Therefore,
(undirected) trees appear as generalizations of words and it is natural to look
forward to count the patterns occurring in it.
Recent work by Crochemore et al. \cite{squaresInTree} showed that the maximum
number of squares in a tree of size $n$ is in $\Theta(n^{4/3})$.  This is
asymptotically bigger than in the case of words, for which the number of squares
is known to be in $\Theta(n)$ \cite{FraenkelSimpson1998}. We discuss here the
number of palindromes and show that, as for squares, the number of palindromes
in trees is asymptotically bigger than in words. Figure \ref{fig:plusDePal}, taken
from \cite{DPV}, shows an example of a tree having more nonempty palindromes
than edges, so that Theorem \ref{thm DrJuPi} does not apply to trees.

\begin{figure}[ht!]
  \begin{center}
    \begin{tikzpicture}[scale=\tikzScale]
      \foreach \x/\a in { 0/-,1/b,2/a,3/a,4/a,5/b } {
        \fill (2*\x,0) circle (\tailleCercle);
        \ifthenelse{\NOT \x=0}{
          \draw (2*\x,0) -- node[above]{$\a$} (2*\x-2,0);
        }
      }
      \fill (6,1.8) circle (\tailleCercle);
      \draw (6,1.8) --node[left]{$b$} (6,0);
    \end{tikzpicture}
  \end{center}
  \caption{A tree $T$ with $6$ edges and  $7$ nonempty palindromes, presented in \cite{DPV}.}
  \label{fig:plusDePal}
\end{figure}
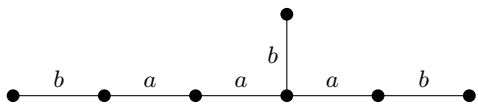

Indeed, the number of nonempty factors in a tree is at most the ways of choosing a couple of edges $(e_i, e_j)$, and these factors correspond to the unique  shortest path from $e_i$ to $e_j$. Therefore, the number of nonempty palindromes in a tree cannot exceed the square of its number of edges. In this article, we exhibit a family of trees with a number of palindromes substantially larger than the bound given by Theorem \ref{thm DrJuPi}. We give a value, up to a constant, for the maximal number of palindromes in trees having a particular language,   and we conjecture that this value holds for any tree.
 
\section{Preliminaries}
Let $\Sigma$ be a finite alphabet, $\Sigma^*$ be the set of finite words over
$\Sigma$, $\varepsilon \in \Sigma^*$ be the empty word and $\Sigma^+ =
\Sigma^*\setminus \{\varepsilon\}$ be the set of nonempty words over $\Sigma$.
We define the \emph{language} of a word $w$ by $\lang(w) = \{f \in \Sigma^* \mid w =
pfs,\ p,s\in \Sigma^*\}$ and its elements are the \emph{factors} of $w$.  
The \emph{reverse} of $w$ is defined by  $\m{w} =
w_{|w|}w_{|w|-1}\ldots w_2w_1$, where $w_i$ is the $i$-th letter of $w$ and
$|w|$, the length of the word. The number of occurrences of a given letter $a$ in the word $w$ is denoted $|w|_a$. A word $w$ is a \emph{palindrome} if $w =
\m{w}$. The restriction of $\lang(w)$ to its palindromes is denoted
$\Pal(w) = \{u \in \lang(w)\mid u = \m{u}\}$. \\

Some notions are issued from graph theory. We consider a \emph{tree} to be an
undirected, acyclic and connected graph. It is well known that the number of
nodes in a tree is exactly one more than the number of edges. The \emph{degree}
of a node is given by the number of edges connected to it. A \emph{leaf} is a
node of degree $1$.  We consider a tree $T$ whose edges are labeled by letters
in $\Sigma$. 
Since in a tree there exists a unique simple path between any pair of nodes,
the function $p(x,y)$ that returns the list of edges along the path from the
node $x$ to the node $y$ is well defined, and so is the sequence $\pi(x,y)$ of
its labels.
The word $\pi(x,y)$ is called a \emph{factor} of $T$ and the set of all its
factors, noted $\lang(T) = \{\pi(x,y) \mid x,y \in \text{Nodes}(T)\}$, is
called the \emph{language} of $T$.  As for words, we define the palindromic
language of a tree $T$ by $\Pal(T) = \{w \in \lang(T) \mid w = \m{w}\}$.
Even though the \emph{size} of a tree $T$ is usually defined by its nodes, we define it here to be  the number of its edges and denote it by $|T|$. This emphasizes the analogy with words, where the length is defined by the number of letters. 
Observe that, since a nonempty path is determined by its first and last edges,
the size of the language of $T$ is bounded by:
\begin{equation}\label{eq:nbMots}
  \lang(T) \leq |T|^2+1.
\end{equation}

Using the definitions above, we can associate a threadlike tree $W$ to a pair of words $\{w, \m{w}\}$. We may assume that $x$ and $y$ are its extremal nodes (the leaves). Then, $w = \pi(x,y)$ and $\m{w} = \pi(y,x)$. The size of $W$ is equal to $|w|=|\m{w}|$. Analogously, $\Pal(W) = \Pal(w) = \Pal(\m{w})$. The language of $W$ corresponds to the union of the languages of $w$ and of $\m{w}$.
For example, Figure \ref{fig:arbreLineaire} shows the word $ababb$ as a
threadlike tree. Any factor of the tree is either a factor of $\pi(x,y)$,
if the edges are read from left to right, or a factor of $\pi(y,x)$, otherwise.

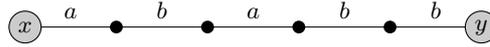
\begin{figure}[h!]
  \centering
  \begin{tikzpicture}[scale=\tikzScale]
    \foreach \x/\a in { 1/a,2/b,3/a,4/b,5/b } { \draw (2*\x-2,0) --node[above]{$\a$} (2*\x,0); }
    \draw[fill=gray!40] (0,0) circle (\tailleNoeud*(1.0/\tikzScale) node {$x$};
    \draw[fill=gray!40] (10,0) circle (\tailleNoeud*(1.0/\tikzScale) node {$y$};
    \foreach \x/\a in { 1,2,3,4 } { \fill (2*\x,0) circle (\tailleCercle); }
  \end{tikzpicture}
  \caption{A threadlike tree represents a pair formed by a word and its reverse.\label{fig:arbreLineaire}}
\end{figure}

For a given word $w$, we denote by $\Delta(w)$ its run-length-encoding, that is the sequence of constant block lengths. For example, for the French word ``appelle'',  $\Delta(\text{appelle}) = 12121$. As well, for the sequence of integers $w=11112111211211$, $\Delta(w)=4131212$. Indeed, each letter of  $\Delta(w)$ represents the length of a block, while the length of $\Delta(w)$ can be associated with the number of blocks in $w$.

Given a fixed alphabet $\Sigma$, we define an infinite sequence of families of trees
\[
  \T_k = \{\text{tree } T \mid |\Delta(f)| \leq k \textrm{ for all } f \in \lang(T) \}.
\]

For any positive integer $k$, we count the maximum number of palindromes of
any tree of $\T_k$ according to its size. To do so, we define the function 
\[
  \P_k(n) = \max_{T \in \T_k, |T| \leq n} |\Pal(T)|.
\]

This value is at least equal to $n+1$. It is known \cite{DrJuPi} that each
prefix $p$ of a Sturmian word contains $|p|$ nonempty palindromes. This
implies that \mbox{$\P_\infty(n) \in \Omega(n)$}. On the other hand, equation
(\ref{eq:nbMots}) provides a trivial upper bound on the growth rate of $\P_k(n)$
since it implies \mbox{$\P_\infty(n) \in \bigo(n^2)$}. We point out that $\P_k(n)$ is an increasing function with respect to $k$. In the following sections we
provide the asymptotic growth, in $\Theta$-notation, of $\P_k(n)$, for $k\leq
4$. Although we have not been able to prove the asymptotic growth for $k\geq
5$, we explain why we conjecture that
\mbox{$\P_\infty(n) \in \Theta( \P_4(n) )$} in section \ref{sec:hypothese}.

\section{Trees of the family $\T_2$}\label{t1}
First recall that, by definition, every nonempty factor of a tree $T$ in $\T_2$ has either one or two blocks of distinct letters. In other terms, up to a renaming of the letters, every factor in $T$ is of the form $a^*b^*$. Therefore, any palindrome in $T$ is on a single letter.
From this, we can deduce a value for $\P_2(n)$ :
\begin{prop}
    The maximal number of palindromes for the family $\mathcal T_2$ is
    $\P_2(n) = n+1$.
\end{prop}

\begin{proof}
    The number of nonempty palindromes on a letter $a$ is the length of the longest factor containing only $a$'s. Thus, the total number of palindromes is at most the number of edges in $T$, plus one (for the empty word). This leads directly to $\P_2(n) \leq n+1$.  On the other hand, a word of length $n$ on a single-letter alphabet contains $n+1$ palindromes. This word is associated to threadlike tree in $\T_1$. Therefore, $\P_2(n) = n+1$.
\end{proof}

\section{Trees of the families $\T_3$ and $\T_4$}
In this section, we show that $\{\P_3(n), \P_4(n)\} \subseteq
\Theta(n^\frac{3}{2})$. To do so, we proceed in two steps. First, we present a
construction that allows to build arbitrary large trees in $\T_3$ such that the
number of palindromes in their languages is large enough to show that $\P_3(n)
\in \Omega(n^\frac{3}{2})$. Then, we show that, up to a constant, this
construction is optimal for all trees of $\T_3$ and $\T_4$.

\subsection{A lower bound for $\P_3(n)$.}

\subsubsection{Some elements from additive combinatorics.}\label{comb add}
An integer sequence is a \emph{Sidon set} if the sums (equivalently, the
differences) of all distinct pairs of its elements are distinct. There exists
infinitely many of these sequences. For example, the powers of $2$ are an
infinite Sidon set.
The maximal size of a Sidon set $A \subseteq \{1,2, \ldots, n\}$ is only known
up to a constant \cite{Gowers}. This bound is easily obtained since $A$ being
Sidon set, there are exactly $\frac{|A|(|A|+1)}{2}$ sums of pairs of elements
of $A$ and all their sums are less or equal to $2n$. Thus,

\[\frac{|A|(|A|+1)}{2} \leq 2n\]
and $ |A| \leq 2\sqrt{n}$.
Erd\H{o}s and Turán \cite{ET} showed that for any prime number $p$, the
sequence
\begin{align}  \label{eq:Erdos}
  A_p = (2pk + (k^2\mod{p}))_{k=1,2,\ldots, p-1},
\end{align}
is a Sidon set. The reader should notice that, since there exists arbitrarily
large prime numbers, there is no maximal size for sequences constructed in this
way.  \\

Moreover, the sequence $A_p$ is, up to a constant, the densest possible. Indeed, 
the maximum value of any element of $A_p$ is less than $2p^2$ and $|A_p| =
p-1$. Since a Sidon set in $\{1,2,\ldots,n\}$ is of size at most $2\sqrt{n}$,
the density of $A_p$ is $\sqrt{8}$ (around $2.83$) times smaller, for any large
$p$.

\subsubsection{The hair comb construction.} \label{sssec peigne}
Our goal is to describe a tree having a palindromic language of size substantially larger than the size of the tree. In this section, we build a tree $\mathcal{C}_p \in \T_3$ for any prime $p$ containing a number of palindromes in $\Theta(|\C_p|^\frac{3}{2})$.\\

For each prime number $p$, let $B = (b_1,\dots,b_{p-2})$ be the sequence
defined by $b_i = a_{i+1} - a_i$,  where the values $a_i$ are taken in the sequence $A_p$ presented above, equation \eqref{eq:Erdos}, and let $\C_p$ be the tree constructed as
follows :
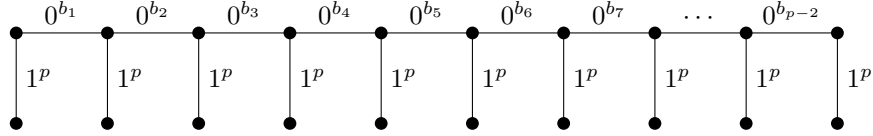
\begin{center}
  \begin{tikzpicture}[scale=\tikzScale]
    \begin{scope}[shift={(0,0)}]
      \def\n{9}
      \def\m{7}
      \foreach \x in { 0,...,\n } {
        \def\y{\x+1}
        \node (H-\x) at (2*\x,0) {};
        \node (B-\x) at (2*\x,-2) {};
        \fill (2*\x,0) circle (\tailleCercle);
        \fill (2*\x,-2) circle (\tailleCercle);
        \draw (2*\x,-2) -- node[right]{$1^p$} (2*\x,0);
      }
      \foreach \x in { 1,...,\m } {
        \draw (2*\x-2,0) -- node[above]{$0^{b_\x}$} (2*\x,0);
      }
    \end{scope}
    \draw (14,0) edge [above] node {$\cdots$}(16,0);
    \draw (16,0) edge [above] node {$0^{b_{p-2}}$}  (18,0);
  \end{tikzpicture}
\end{center}

\begin{prop}\label{caract suite B}
  The sums of the terms in each contiguous subsequence of $B$ are pairwise distinct.
\end{prop}
\begin{proof}
  By contradiction, assume that there exists indexes $k,l,m,n$ such
  that $\sum_{i=k}^{l} b_i = \sum_{j=m}^{n} b_j$. By definition of $B$,
  \[\sum_{i=k}^{l} b_i = \sum_{i=k}^{l}( a_i - a_{i-1}) = a_l - a_{k-1}
    \textrm{ and } \sum_{j=m}^{n} b_j = a_n - a_{m-1}.\]
    This implies that $a_l+a_{m-1} = a_n+a_{k-1}$, which is impossible.
  \end{proof}
\begin{lem}\label{pal de cp}
  The number of palindromes in $\C_p$ is in $\Theta(p^3)$.
\end{lem}
\begin{proof}
  The  nonempty palindromes of $\C_p$ are of three different forms. Let $c_0$ be the
  number of palindromes of the form $0^+$, $c_1$ be the number of palindromes
  of the form $1^+$ and $c_{101}$ be the number of palindromes of the form
  $1^+0^+1^+$. The number of palindromes of $\C_p$ is clearly $|\pal(\C_p)| =
  c_0 + c_1 + c_{101} + 1$, where one is added for the empty word. 
  \begin{align*}
  c_0 &= b_1 + b_2 + \cdots + b_{p-2} = a_{p-1} - a_1 = 2p^2-4p,\\
    c_1 &= p,\\
    c_{101} &= \left|\{ 1^x0^y1^x \in \pal(\C_p) \} \right| \\
    &= |\{x \mid 1 \leq x \leq p\}| \cdot 
       |\{ y \mid y = \textstyle \sum_{i=k}^{l} b_i \textrm{ for } 1\leq k \leq l \leq p-2 \}|\\ 
       & = {\textstyle\frac{1}{2}}p(p-1)(p-2).
   \end{align*}
   The last equality comes from the fact that there are $(p-1)(p-2)/2$
   possible choices of pairs $(k,l)$ and proposition \ref{caract suite B}
   guarantees that each choice sums up to a different value. The asymptotic
   behavior of the number of palindromes is determined by the leading term $p^3$.
\end{proof}    
\begin{lem}\label{taille de C_p}
  The number of edges in $\C_p$ is in $\Theta(p^2)$.
\end{lem}
\begin{proof}
  The number of edges labeled by $0$ is $b_1 + b_2 + \ldots + b_{p-2} = 2p^2-4p$.
  For those labeled with $1$, there are exactly $p-1$ sequences of edges
  labeled with $1$'s and they all have length $p$. The total number of
  edges is thus $2p^2-4p+p(p-1) = 3p^2-5p$.
\end{proof}

\begin{thm}
  $\P_3(n) \in \Omega(n^\frac{3}{2})$.
\end{thm}

\begin{proof}
  Lemmas \ref{pal de cp} and \ref{taille de C_p} implies that the number of
  palindromes in $\C_p$ is in $\Theta(|\C_p|^{\frac{3}{2}})$. Since
  there are infinitely many trees of the form $\C_p$ and since their size
  is not bounded, these trees provide a lower bound on the growth rate of
  $\P_3(n)$.
\end{proof}
\subsection{The value of $\P_4(n)$ is in  $\Theta(n^{\frac{3}{2}})$.}
In this subsection, we show that the asymptotic value of $\P_3(n)$ is reached by the hair comb construction, given above, and that it is the same  value for $\P_4(n)$.
\begin{thm} \label{valeur "max" de r2}
    $\P_4(n) \in \Theta(n^\frac{3}{2})$.
\end{thm}

Before giving a proof of this theorem, we need to explain some arguments.
We first justify why we reduce any tree of $\T_4$ to a tree in $\T_3$. Then, we present some properties of the latter trees in order to establish an upper bound on $\P_4(n)$.

\begin{lem}\label{reduction a Q}
    For any $T \in \T_4$, there exists a tree $S \in \T_3$ on a binary alphabet satisfying $|S| \leq
    |T|$, and with \mbox{$\frac{1}{|\Sigma|^2}|\Pal(T)|- |T| \leq |\Pal(S)| \leq |\Pal(T)|$.}
\end{lem}

\begin{proof}
    If there is in $T$ no factor with three blocks starting and ending with the same letter, this means that all the palindromes are repetitions of a single letter. We then denote by $a$ the letter on which the longest palindrome is constructed. It might not be unique, but it does not matter. Let $S$ be the longest path labeled only with $a$'s. Then, $ |\Pal(T)|\leq |\Sigma||\Pal(S)|\leq |\Sigma||\Pal(T)|$.\\
    Otherwise, let $a$ and $b$ be letters of $\Sigma$ and let $(a,b)$ be a pair of letters for which $|\lang(T) \cap \Pal(a^+b^+a^+)|$ is maximal. We define the set
    \[E_S=\cup \big(\ p(u,v) \mid \pi(u,v) \in \Pal(a^+b^+a^+)\ \big)\]
    and let $S$ be the subgraph of $T$ containing exactly the edges of $E_S$ and the nodes connected to these edges. 
    Then, there are three things to prove :
    \begin{itemize}
        \item $S$ is a tree: Since $S$ is a subgraph of $T$, it cannot contain any cycle. We however need to prove that $S$ is connected. To do so, assume that $S$ has two connected components named $C_1$ and $C_2$. Of course, $\lang(C_1) \subseteq a^*b^*a^*$ and $C_1$ has at least one factor in $a^+b^+a^+$. The same holds for $C_2$.
        Since $T$ is a tree, there is a unique path in $T\backslash S$ connecting $C_1$ and $C_2$. We call it $q$.\\
        There are  paths in $C_1$ and in $C_2$ starting from an extremity of
        $q$ and containing factors in $b^+a^+$. Thus, by stating that $w$ is
        the trace of $q$, $T$ has a factor $f \in a^+b^+a^*wa^*b^+a^+$. By
        hypothesis, $T \in \T_4$ so any factor of $T$ contains at most four
        blocks. Then, $f$ has to be in $a^+b^+wb^+a^+$, with $w \in b^*$ and so $q$ is a
        path in $S$. A contradiction. \\      
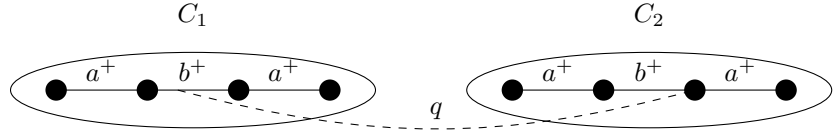
        \begin{center}
            \begin{tikzpicture}
                \begin{scope}[shift={(0,0)}]
                \foreach \x in { 0,...,3 } {
                    \def\y{\x+1}
                    \node (H-\x) at (1.2*\x,0) {};
                    \fill (1.2*\x,0) circle (\tailleCercle);
                }
                \foreach \x in { 5,...,8 } {
                    \def\y{\x+1}
                    \node (H-\x) at (1.2*\x,0) {};
                    \fill (1.2*\x,0) circle (\tailleCercle);
                }
                \end{scope}
                \draw (0,0) edge [above] node {$a^+$}  (1.2,0);
                \draw (1.2,0) edge [above] node {$b^+$}  (2.4,0);
                \draw (2.4,0) edge [above] node {$a^+$}  (3.6,0);
                \node (C1) at (1.8,0) {};
                \node (C2) at (7.8,0) {};
               
                \draw (6,0) edge [above] node {$a^+$}  (7.2,0);
                \draw (7.2,0) edge [above] node {$b^+$}  (8.4,0); 
                \draw (1.8,0) ellipse (2.4cm and 0.5cm) node [above of=C1] {$C_1$};
                \draw (7.8,0) ellipse (2.4cm and 0.5cm) node [above of=C2] {$C_2$};
                \draw (8.4,0) edge [above] node {$a^+$}  (9.6,0);
                \path[->] 
                (1.6,0) edge [dashed, bend right=15] node [above]{$q$} (8.4,0);
            \end{tikzpicture}
        \end{center}
        \item $S \in \T_3$ is on a binary alphabet: 
        By construction, $S$ contains only edges labeled by $a$ or $b$ and has no leaf connected to an edge labeled by $b$. This implies that if $S$ contains a factor $f \in a^+b^+a^+b^+$, $f$ may be extended to $f' \in a^+b^+a^+b^+a^+$, which does not appear in $T$.
        \item $|\Pal(S)| \geq \frac{1}{|\Sigma|^2}|\Pal(T)|-|T|$:  We chose $(a,b)$ to be the pair of letters for which the number of palindromes on an alphabet of size at least $2$ was maximal. The number of palindromes on a single letter is at most $|T|$. Thus, 
        \[\frac{1}{|\Sigma|^2}|\Pal(T)|-|T| \leq |\Pal(S)| \leq |\Pal(T)| .\]
    \end{itemize}
\end{proof}

\begin{lem}\label{101 pas 010}
    For any $T \in \T_3$, $T$ cannot contain both factors of $0^+1^+0^+$ and of $1^+0^+1^+$.
\end{lem}
\begin{proof}
    We proceed by contradiction. Assume that there exists in $T$ four nodes $u,v,x,y$ such that $\pi(u,v)\in 0^+1^+0^+$ and $\pi(x,y) \in 1^+0^+1^+$. Since $T$ is a tree, there exists a unique path between two nodes. In particular, there is a path from $w \in \{u,v\}$ to  $w' \in \{x,y\}$ containing a factor of the form $ 0^+1^+0^+\Sigma^*1^+$, which contradicts the hypothesis that $T \in \T_3$. 
\end{proof}

    We now define the restriction $\R_a(T)$ of a tree $T$ to the letter $a$ by keeping from $T$ only the edges labeled by $a$ and the nodes connected to them.
\begin{lem}\label{arbre a connexe}
    Let $T$ be in $\T_3$. There exists at least one letter $a \in \Sigma$ such that $\R_a(T)$ is connected.
\end{lem}
\begin{proof}
    If $T$ does not contain a factor on at least two letters that starts and
    ends with the same letter, that is of the form $b^+a^+b^+$, then $\R_a(T)$
    is connected for any letter
    $a$.\\
    Otherwise, assume that a factor $f \in b^+a^+b^+$ appears in $T$. Then,
    $\R_a(T)$ must be connected. By contradiction, suppose there exists an edge
    labeled with $a$ that is connected to the sequence of $a$'s in
    $f$, by a word $w$ that contains another letter than $a$.
    Then, there exists a word of the form $awa^+b^+$ in $\lang(T)$ and this contradicts
    the hypothesis that $T\in \T_3$.
\end{proof}\\
Given a node $u$ in a tree, we say that $u$ is a \emph{splitting on the letter
$a$} if $\deg(u) \geq 3$ and there is at least two edges labeled with $a$
connected to $u$.
\begin{lem}\label{1-1}
    Let $T$ be in $\T_3$.
    Then, there is a tree $T'$ of size $|T|$ such that \mbox{$\lang(T) \subseteq \lang(T')$} and there exists a letter $a \in \Sigma$ such that any splitting of $T'$ is on the letter $a$.
\end{lem}

\begin{proof}
  If $T$ is in $\T_2$, we apply the upcoming transformation to every branches.
  Otherwise, assume that a factor of the form $b^+a^+c^+$ appears in $T$ (note that $b$ might be equal to $c$). We allow splittings only on the letter $a$.
  Let $v$ be a node of $T$ that is a splitting on $b \in \Sigma \backslash \{a\}$ (if it does not exist, then $T' = T$). By the hypothesis on $T$, this means that there exists, starting from $v$, at least two paths labeled only with $b$'s leading to leaves $x$ and $y$.
  \begin{center}
    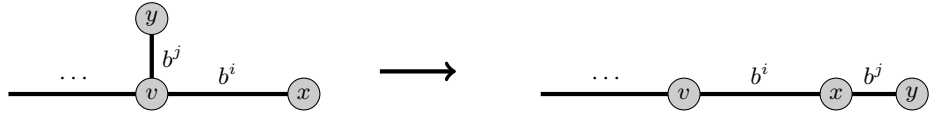
\begin{figure}
      \begin{minipage}{0.4\linewidth}
        \begin{tikzpicture}
          \begin{scope}
            \node (v) at (0,0) {};
            \node (x) at (2,0) {};
            \node (y) at (0,1) {};
            \node (A) at (-2,0) {};
            \draw[ultra thick] (v) edge node[above]{$b^i$}(x);
            \draw[ultra thick] (v) edge node[right]{$b^j$}(y);
            \draw[ultra thick] (v) edge node[above]{$\cdots$}(A);
            \draw[fill=gray!40] (v) circle (\tailleNoeud) node {$v$};
            \draw[fill=gray!40] (x) circle (\tailleNoeud) node {$x$};
            \draw[fill=gray!40] (y) circle (\tailleNoeud) node {$y$};
            \draw[ultra thick,->] (3,0.3) -- ++(1,0);
          \end{scope}
          \begin{scope}[shift={(7,0)}]
            \node (v) at (0,0) {};
            \node (x) at (2,0) {};
            \node (y) at (3,0) {};
            \node (A) at (-2,0) {};
            \draw[ultra thick] (v) edge node[above]{$b^i$}(x);
            \draw[ultra thick] (x) edge node[above]{$b^j$}(y);
            \draw[ultra thick] (v) edge node[above]{$\cdots$}(A);
            \draw[fill=gray!40] (v) circle (\tailleNoeud) node {$v$};
            \draw[fill=gray!40] (x) circle (\tailleNoeud) node {$x$};
            \draw[fill=gray!40] (y) circle (\tailleNoeud) node {$y$};
          \end{scope}
        \end{tikzpicture}
      \end{minipage}
      \caption{The destruction of a splitting on the letter $b$.}
    \end{figure}
  \end{center}
    We assume that $|\pi(v, x)| \geq |\pi(v,y)|$. Then, the words having $\pi(v,y)$ as suffix are a subset of those for which $\pi(v,x)$ is suffix. Therefore, the only case where $\pi(v,y)$ may contribute to the language of $T$ is when both the edges of $\pi(v,x)$ and $\pi(v,y)$ are used. The words of this form are composed only of $b$'s and are of length at most $|\pi(v, x)| + |\pi(v,y)|$.
  Moving the edges between $s$ and $y$ to the other extremity of $x$, we construct a tree for which the language contains $\lang(T)$ and having the same number of nodes.
  Finally, we can apply this procedure until the only remaining splittings are on the letter $a$. This leads to $T'$.
\end{proof}

We are now ready to prove the main theorem.
\begin{proof} [Theorem \ref{valeur "max" de r2}: $\P_4(n) \in \Theta(n^\frac{3}{2})$.] 
  Let $T$ be in $\T_4$.
  By assumption, each factor of $T$ contains at most four blocks of distinct letters. 

  \paragraph{1.~}Let $S \in \T_3$ be such that $|S|\leq |T|$, $\lang(S) \subseteq \{0,1\}^*$ and \mbox{ $ \frac{|\Pal(T)|-|T|}{|\Sigma|^2} \leq |\Pal(S)| \leq |\Pal(T)|$}. Using lemma \ref{reduction a Q}, we know that this exists.\\
    We know by lemma \ref{101 pas 010} that $S$ may contain factors in $1^{+}0^{+}1^{+}$, but not in $0^{+}1^{+}0^{+}$.
    \paragraph{2.~}By lemma \ref{1-1}, there exists a tree $S'$ with $|S'|=|S|$, such that $\lang(S)\subseteq \lang(S')$, and with no splitting on the letter $1$.
    \paragraph{3.~}Finally, we count the palindromes in $S'$. The form of these
    palindromes is either $0^+$, $1^+$ or $1^+0^+1^+$. For the palindromes on
    a one-letter alphabet, their number is bounded by $n$, where $n$ is the
    size of $S'$. We now focus on the number of palindromes of the form
    $1^+0^+1^+$. Call $c_{101}$ this number. We show that \mbox{ $c_{101} \leq 2n\sqrt{n}$}.

      Since $S'$ does not admit any splitting on the letter $1$, each 
      connected component of $R_1(S')$ is a threadlike branch going from a leaf
      of $S'$ to a node of $R_0(S')$. We name these connected components $b_1,
      \ldots, b_m$ and by lemma \ref{arbre a connexe}, we know that $R_0(S')$ is connected.

    Let $b_i$ and $b_j$ be two distinct branches of $S'$. By abuse of notation,
    we note $\pi(b_i,b_j)$ the word defined by the unique path from $b_i$ to
    $b_j$. Let $l$ be such that $\pi(b_i,b_j) = 0^l$ and suppose that $|b_i|
    \leq |b_j|$. Then, for any node $u$ in $b_i$, there exists a unique node
    $v$ in $b_j$, such that the word $\pi(u,v) = 1^k 0^l 1^k$ is a palindrome.
    Moreover, if $|b_i| < |b_j|$, then there are nodes in $b_j$ that cannot
    be paired to a node of $b_i$ in order to form a palindrome. From this
    observation, a first upper bound is:
    \begin{equation}\label{eq_borne_1}
      c_{101} \leq \sum_{1\leq i < j \leq m} \min( |b_i|, |b_j| ).
    \end{equation}
    Another way to bound $c_{101}$ is to count the palindromes of the form
    $1^+0^+1^+$ according to the length of the block of $0$'s. For each length $l$ from $1$
    to $n$, there might be more than one pair $\{b_i,b_j\}$ that produces
    palindromes with central factor $0^l$. This provides a second upper
    bound: 
    \begin{equation}\label{eq_borne_2}
      c_{101} \leq \sum_{l=1}^n \max_{\begin{array}{c}\scriptstyle 1\leq i< j\leq m \\
      \scriptstyle \pi(b_i,b_j) = 0^l\end{array}} \left( \min(|b_i|,|b_j|) \right)
    \end{equation}

    In order to obtain the desired bound on $c_{101}$ we combine these
    two bounds. Let $ B' = \{ i \mid |b_i| \geq \sqrt{n} \}$.
    Since $n$ is the size of $S'$, we have that $|B'| \leq \sqrt{n}$ and
    that the average size of the branches $b_i$ is such that $i \in B'$ is bounded
    by $n/|B'|$. By applying the bound from $(\ref{eq_borne_1})$ to the
    palindromes formed by two branches in $B'$, we obtain that the number of
    such palindromes is:
    \begin{equation}\label{eq_count_pal_1}
      \sum_{\begin{array}{c}\scriptstyle 1\leq i<j \leq m \\ \scriptstyle \{i,j\} \subseteq B'\end{array}} 
      \min( |b_i|,|b_j| ) \leq \frac{|B'|(|B'|-1)}{2} \frac{n}{|B'|} \leq n\sqrt{n}.
    \end{equation}

    Finally, it remains to count the number of palindromes that are defined by
    pairs of branches $\{b_i, b_j\}$ such that $i$ or $j$ is not in $B'$. In such
    case, we always find that $\min(|b_i|,|b_j|) < \sqrt{n}$. The number of
    such palindromes is:
    \begin{equation}\label{eq_count_pal_2}
      \sum_{l=1}^n \max_{\begin{array}{c}\scriptstyle 1 \leq i< j \leq m\\\scriptstyle \pi(b_i,b_j) = 0^l\\ \scriptstyle
        \{i,j\} \not\subset B' \end{array}} 
      \left( \min(|b_i|,|b_j|) \right)
      < n \sqrt{n}.
    \end{equation}
    Since each palindrome in $S'$ is counted by equation
    $(\ref{eq_count_pal_1})$ or $(\ref{eq_count_pal_2})$,
    we obtain, summing both, $c_{101} < 2n\sqrt{n} = 2|S'|^\frac{3}{2}$.
    We deduce that, for any tree $T$ in $\T_4$, the number of palindromes is bounded by
    \[|\Pal(T)| \leq |\Sigma|^2|\Pal(S)|+|T| <  2|\Sigma|^2|S'|^{\frac{3}{2}}+|T| \leq 2|\Sigma|^2|T|^{\frac{3}{2}}+|T|.\]
    Using the fact that the alphabet is fixed (so its size is given by a constant), it is enough to prove that $\P_4(n) \in \bigo(n^{\frac{3}{2}})$. Combining this result with the one given in section \ref{sssec peigne}, one may assert that both $\P_3(n)$ and  $\P_4(n)$ are in $\Theta(n^{\frac{3}{2}})$.
\end{proof}

\section{Hypotheses for the construction of trees with a lot of distinct palindromes}\label{sec:hypothese}
Let $T$ be a tree that maximizes the number of palindromes for its size. It is
likely that $T$ contains triples of nodes $(u,v,w)$ such that $\pi(u,v)$,
$\pi(u,w)$ and $\pi(v,w)$ are all palindromes. Suppose it is the case, and
define $T'$ as the restriction of $T$ to the paths that join $u$, $v$ and $w$.
We have that either $T'$ is a threadlike tree, or $T'$ has three leaves and a unique node of degree $3$. The first case is of no interest here since it
is equivalent to words, while the latter case implies a restrictive structure on
the factors $\pi(u,v)$, $\pi(u,w)$ and $\pi(v,w)$. We now focus on the second case
and call $x$ the unique node of $T'$ with degree $3$. 

Let $U = \pi(u,x)$, $V = \pi(v,x)$,  $W = \pi(w,x)$ and, without loss of
generality, suppose that $|U| \leq |V| \leq |W|$. Then, as shown in Figure
\ref{fig:structureTPrime}, $U\m{V}$, $U\m{W}$ and $V\m{W}$ are all
palindromes.

\begin{figure}[H]
  \begin{center}
    \begin{tikzpicture}[scale=1]
      \def\shift{0.25}
      \node (v) at (-1,0) {};
      \node (w) at (7,0)  {};
      \node (x) at (2,0)  {};
      \node (u) at (2,1)  {};
      \node (v') at (0,0) {};
      \node (w') at (6,0) {};
      \node (w'') at (4,0) {};
      \draw[->] ($(u.south)+(\shift,0)$) --node[right]{$U$} ($(x.north)+(\shift,0)$);
      \draw[->] ($(v.east)+(0,\shift)$) --node[above]{$V$} ($(x.west)+(0,\shift)$);
      \draw[->] ($(v.east)+(0,-\shift)$) --node[below]{$U$} ($(v'.west)+(0,-\shift)$);
      \draw[->] ($(v'.east)+(0,-\shift)$) --node[below]{$A$} ($(x.west)+(0,-\shift)$);
      \draw[->] ($(w.west)+(0,\shift)$) --node[above]{$W$} ($(x.east)+(0,\shift)$);
      \draw[->] ($(w.west)+(0,-\shift)$) --node[below]{$U$} ($(w'.east)+(0,-\shift)$);
      \draw[->] ($(w'.west)+(0,-\shift)$) --node[below]{$A$} ($(w''.east)+(0,-\shift)$);
      \draw[->] ($(w''.west)+(0,-\shift)$) --node[below]{$B$} ($(x.east)+(0,-\shift)$);
      \fill (v')  circle (\tailleCercle);
      \fill (w')  circle (\tailleCercle);
      \fill (w'') circle (\tailleCercle);
      \draw[ultra thick] (v) edge (x);
      \draw[ultra thick] (w) edge (x);
      \draw[ultra thick] (u) edge (x);
      \draw[fill=gray!40] (u) circle (\tailleNoeud) node {$u$};
      \draw[fill=gray!40] (v) circle (\tailleNoeud) node {$v$};
      \draw[fill=gray!40] (w) circle (\tailleNoeud) node {$w$};
      \draw[fill=gray!40] (x) circle (\tailleNoeud) node {$x$};
    \end{tikzpicture}
  \end{center}
  \caption{The structure of the tree $T'$. The palindromicity of $U\m{V}$,
  $U\m{W}$ and $V\m{W}$ forces that $V$ starts with $U$ while $W$ starts with
  both factors $U$ and $V$. \label{fig:structureTPrime}}
\end{figure}
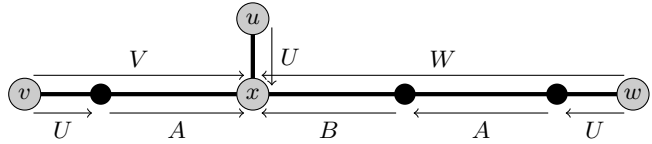

Let $A$ be the suffix of length $|V|-|U|$ of $V$. Since, by hypothesis,
$U\m{V}$ is a palindrome, $V=UA$ and $A$ is a palindrome.
Similarly, let $B$ be the suffix of length $|W|-|V|$ of W. This implies that
$W=VB=UAB$ and both $B$ and $AB$ are palindromes. Using a well-known lemma from Lothaire \cite{lothaire1}, we prove that $AB$ is
periodic.
\begin{lem}[\rm Proposition 1.3.2 in \cite{lothaire1}]\label{lothaire}
	Two words commute if and only if they are powers of the same word.
\end{lem}

The next proposition states that the word $ABA$ is periodic and that its period
is at most the gcd of the difference of length of the three paths between $u$,
$v$ and $w$. More formally, let  
\[
  p = \gcd\left( |\pi(u,w)|-|\pi(u,v)|, |\pi(v,w)|-|\pi(u,v)|, |\pi(v,w)|-|\pi(u,w)| \right).
\]
\begin{prop}
  There exists a word $S$ and two integers $i,j$ such that $|S|$
  divides $p$ and $A = S^i$ and $B = S^j$.
\end{prop}
\begin{proof}
  Since $A$, $B$ and $AB$ are palindromes,  $AB = \m{AB} = \m{B} \m{A} = BA$.
  Thus, by lemma \ref{lothaire}, there exists a word $S$ such that $A=S^i$ and
  $B=S^j$. This implies that $|S|$ divides $\gcd(|A|,|B|)$ and, by construction,
  $\gcd(|A|,|B|) = p$.
\end{proof}

From the above proposition, we deduce that a triple of
nonaligned nodes with any path from a node to another being a
palindrome forces a local structure isomorphic to that of the hair comb tree,
as illustrated in Figure \ref{fig:structureHairComb}.

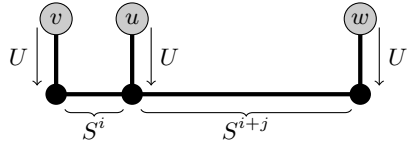
\begin{figure}[H]
  \begin{center}
    \begin{tikzpicture}[scale=1.0]
      \def\shift{0.25}
      \node (u) at (2,1) {};
      \node (v) at (1,1) {};
      \node (w) at (5,1) {};
      \node (x) at (2,0) {};
      \node (y) at (1,0) {}; \node (z) at (5,0) {};
      \fill (x) circle (\tailleCercle); 
      \fill (y) circle (\tailleCercle); 
      \fill (z) circle (\tailleCercle);
      \draw[ultra thick] (u) edge (x);
      \draw[ultra thick] (v) edge (y);
      \draw[ultra thick] (w) edge (z);
      \draw[ultra thick] (y) edge (z);
      \draw[->] ($(u.south)+(\shift,0)$) --node[right]{$U$} ($(x.north)+(\shift,0)$);
      \draw[->] ($(v.south)+(-\shift,0)$) --node[left]{$U$} ($(y.north)+(-\shift,0)$);
      \draw[->] ($(w.south)+(\shift,0)$) --node[right]{$U$} ($(z.north)+(\shift,0)$);
      \draw [decoration={brace, mirror, raise=0.2cm}, decorate] (y) -- (x)
        node [pos=0.5,anchor=north,yshift=-0.2cm] {$S^i$}; 
      \draw [decoration={brace, mirror, raise=0.2cm},decorate] (x) -- (z) node
        [pos=0.5,anchor=north,yshift=-0.2cm] {$S^{i+j}$}; 

        \draw[fill=gray!40] (u) circle (\tailleNoeud) node {$u$};
        \draw[fill=gray!40] (v) circle (\tailleNoeud) node {$v$};
        \draw[fill=gray!40] (w) circle (\tailleNoeud) node {$w$};
      \end{tikzpicture}
  \end{center}
  \caption{A triple of nodes with palindromes between each pair of them is
    isomorphic to a part of a hair comb.\label{fig:structureHairComb}
  }
\end{figure}

In a more general way, suppose that a tree contains $m$ leaves $(u_i)_{1 \leq i
\leq m}$, and that each $\pi(u_i,u_j)$ is a palindrome. Let $T'$ be the
restriction of this tree to the paths that connect these leaves and, for
each $i$, let $v_i$ be the first node of degree higher than $2$ accessible
from the leaf $u_i$ in $T'$. By applying the above proposition to each triplet 
$(u_i,u_j,u_k)$, for all $i \neq j$, the word $\pi(u_i,u_j)$ is of
the form 
\[
  \pi(u_i,u_j) = U S^+ \m{U},
\]
where $|U| = \min_i( \pi(u_i, v_i) )$ and $|S|$ divides 
$\displaystyle \gcd_{i \neq j, k\neq l} \left( \big| |\pi(u_i,u_j)| -
|\pi(u_k,u_l)| \big| \right)$.

Moreover, in order to maximize the number of palindromes relatively to the size of the tree, we can choose $S$ to be a single letter. This is indeed possible since the only condition on the length of $S$ is that it divides all the differences of lengths between any palindromic path from a leaf to another.\\

This gives a tree analogous to those presented in section \ref{sssec peigne}, $\C_p$, and for which we have established that $|\Pal(\C_p)| \in \Theta(|\C_p|^{\frac{3}{2}})$. Therefore, we conjecture that \mbox{$\P_\infty(n) \in \Theta(n^{\frac{3}{2}})$}.

\bibliographystyle{splncs03.bst}
\bibliography{Bibliographie}

\end{document}